\documentclass[a4paper,10pt]{article}
\usepackage{setspace}
\usepackage{diagrams}
\usepackage{geometry}
\usepackage{epsfig}
\usepackage{amsfonts}
\usepackage{amsthm}
\usepackage{latexsym}
\usepackage{amsmath}
\usepackage{amscd}
\usepackage{amssymb}
\usepackage{enumerate}
\usepackage{graphicx,oldgerm}
\usepackage{authblk}

\theoremstyle{definition}
\newtheorem{thm}{Theorem}[section]

\newtheorem{lem}[thm]{Lemma}

\newtheorem{prop}[thm]{Proposition}
\newtheorem{defi}[thm]{Definition}

\newtheorem{note}[thm]{Notation}
\newtheorem{para}[thm]{}

\DeclareMathOperator{\codim}{\mathrm{codim}}
\DeclareMathOperator{\NL}{\mathrm{NL}}
\DeclareMathOperator{\prim}{\mathrm{prim}}

\DeclareMathOperator{\red}{\mathrm{red}}

\DeclareMathOperator{\Hc}{\mathcal{H}om}

\DeclareMathOperator{\op}{\mathcal{O}_{\mathbb{P}^3}(d)}

\DeclareMathOperator{\p3}{\mathbb{P}^3}

\DeclareMathOperator{\pr}{\mathrm{pr}}

\DeclareMathOperator{\N}{\mathcal{N}}
\DeclareMathOperator{\T}{\mathcal{T}}

\DeclareMathOperator{\I}{\mathcal{I}}
\DeclareMathOperator{\mo}{\mathcal{O}}

\DeclareMathOperator{\NS}{\mathrm{NS}}

\newcommand{\mb}[1]{\mathbb{#1}}
\newcommand{\mc}[1]{\mathcal{#1}}
\newcommand{\mr}[1]{\mathrm{#1}}
\newcommand{\ov}[1]{\overline{#1}}
\title{On a Griffiths-Harris Conjecture}
\author{Ananyo Dan \thanks{The author has been supported by the DFG under Grant 
KL-$2244/2-1$\\ \newline Humboldt Universit\"{a}t zu Berlin, Institut f\"{u}r Mathematik, Unter den Linden $6$, Berlin $10099$.\\ e-mail: dan@mathematik.hu-berlin.de}}
\date{\today}

\begin{document}
\maketitle
\doublespacing

\section{Introduction}
In $1882$, M. Noether claimed the following statement which was later proven by Lefschetz:
For $d \ge 4$, a very general smooth degree $d$ surface $X$ in $\mathbb{P}^3$ has Picard number $\rho(X)=1$.
This motivates the definition of the \emph{Noether-Lefschetz locus}, denoted by $\NL_d$ parametrizing the 
space of smooth degree $d$ surfaces $X$ in $\mathbb{P}^3$ with $\rho(X)>1$.
One of the interesting problems is to understand the geometry of the Noether-Lefschetz locus. By the Lefschetz $(1,1)$-theorem, we can look at an irreducible
component of the Noether-Lefschetz locus locally as a Hodge locus (see \cite[\S $5$]{v5} for more details).
In particular, denote by $U_d \subseteq \mathbb{P}(H^0(\mathbb{P}^3, \op))$ 
the open subscheme parametrizing smooth projective hypersurfaces in $\mathbb{P}^3$ of degree $d$.
Let $\mathcal{X} \xrightarrow{\pi} U_d$ be the corresponding universal family. For a given $F \in U_d$, denote by $X_F$ the surface $X_F:=\pi^{-1}(F)$. 
Let $X \in U_d$ and $U \subseteq U_d$ be a simply connected neighbourhood of $X$ in $U_d$ (under the analytic topology).
Then, $\pi|_{\pi^{-1}(U)}$ induces a variation of Hodge structure $(\mathcal{H}, \nabla)$ on $U$, where $\mathcal{H}:=R^2\pi_*\mathbb{Z} \otimes 
\mathcal{O}_U$ and $\nabla$ is the Gauss-Manin connection. Note that $\mathcal{H}$ defines a local system on $U$ whose fiber over 
a point $F \in U$ is $H^2(X_F,\mathbb{Z})$. Consider a non-zero element $\gamma_0 \in H^2(X_F,\mathbb{Z}) \bigcap H^{1,1}(X_F,\mathbb{C})$
such that $\gamma_0  \not= c_1(\mathcal{O}_{X_F}(k))$ for $k \in \mathbb{Z}_{>0}$. 
This defines a section $\gamma \in (\mathcal{H} \otimes \mathbb{C})(U)$. Let $\overline{\gamma}$ be the image
of $\gamma$ in $\mathcal{H}/F^2(\mathcal{H} \otimes \mathbb{C})$. The Hodge loci corresponding to $\gamma$, denoted $\NL(\gamma)$ is then defined as
\[ \NL(\gamma):=\{G \in U | \overline{\gamma}_G=0\},\]
where $\overline{\gamma}_G$ denotes the value at $G$ of the section $\overline{\gamma}$. For an irreducible component $L \subset \NL_d$ and $X \in L$, general, we can find $\gamma \in H^{1,1}(X,\mb{Z}):= 
H^2(X,\mb{Z}) \bigcap H^{1,1}(X,\mb{C})$ such that $\overline{\NL(\gamma)}=L$ (the closure taken in $U_d$ under Zariski topology).

One of the first results in this direction is due to Green, Griffiths, Voisin and others (\cite{M3, GH, v2}) which states that 
for an irreducible component $L$ of the Noether-Lefschetz locus, that for $d \ge 4$,
\[d-3 \le \codim (L,U_d) \le \binom{d-1}{3}.\] 
The upper bound follows easily from the fact that $\dim H^{2,0}(X)=\binom{d-1}{3}$ for any $X \in U_d$ (see \cite[\S $6$]{v5}).
We say that $L$ is a \emph{general} component if $\codim L=\binom{d-1}{3}$ and \emph{special} otherwise. 
It was proven by Ciliberto, Harris and Miranda \cite{ca1} that for $d \ge 4$, 
the Noether-Lefschetz locus has infinitely many general components and the union of these components is Zariski dense in $U_d$. 
The guiding principle of much work in the area has been the expectation that special components should be due to the presence of low degree curves.
Voisin \cite{v3} and Green \cite{M3} independently 
proved that for $d \ge 5$, $\codim L=d-3$ if and only if $L$ parametrizes surfaces of degree $d$ containing a line. 
If $d-3<\codim L \le 2d-7$ then $\codim L=2d-7$ and $L$ parametrizes the surfaces containing a conic.
Otwinowska \cite{A1} proved that for an integer $b>0$ and $d \gg b$ if $\codim L \le bd$ then $L$ parametrizes surfaces containing a curve of 
degree at most $b$.

For $r \ge 3$, we define the level $r$-\emph{Noether-Lefschetz locus}, denoted $\NL_{r,d}$ to be the space parametrizing surfaces with 
Picard number greater than or equal to $r$. It has been conjectured by Griffiths and Harris in \cite{GH} that for $r < d$, an irreducible component of $\NL_{r,d}$ is 
of codimension greater than or equal to $(r-1)(d-3)-\binom{r-3}{2}$. Furthermore, the component of $\NL_{r,d}$ parametrizing
surfaces containing $r-1$ coplanar lines is of this codimension. In this article we prove (in Theorem \ref{gh6}) that:
\begin{thm}\label{ite1}
 Let $r \ge 3$ and $d \gg  r$. Let $L$ be an irreducible component of $\NL_{r,d}$. Then, $\codim L \ge (r-1)(d-3)-\binom{r-3}{2}$.
Furthermore, there exists a component $L$ of $\NL_{r,d}$ of this codimension parametrizing surfaces containing $r-1$ coplanar lines.
\end{thm}

The techniques used to prove this result is a combination of deformation theory and Hodge theory. Instead of looking at the Hodge locus corresponding to a Hodge class, we study the Hodge locus corresponding to
a $\mb{Z}$-module of Hodge classes. We then use a result due to Otwinowska, \cite[Theorem $1$]{ot}, to show that if the codimension of an irreducible component $L$ of $\NL_{r,d}$
is less than or equal to $(r-1)(d-3)-\binom{r-3}{2}$, then for a general $X \in L$ there exists a lattice $\Lambda \subset H^{1,1}(X,\mb{Z})$ generated by classes of curves of degree less than or
equal to $r-1$ such that $L$ is locally of the form $\NL(\Lambda)$ (see Proposition \ref{gh14}), where $\NL(\Lambda)$ is the intersection of $\NL(\gamma)$ for all $\gamma \in \Lambda$,
$\bigcap_{\gamma \in \Lambda} \NL(\gamma)$.

We now use the theory of semi-regularity as introduced in \cite{b1} to reduce the problem to a question in flag Hilbert schemes. 
First to fix some notations, for a Hilbert polynomial $P$ for some curve $C$ in $\p3$, we denote by $H_P$ the corresponding Hilbert scheme, parametrizing curves (schemes with pure dimension $1$) with 
Hilbert polynomial $P$.
Throughout this article we denote by $Q_d$ the Hilbert polynomial of a degree $d$
surface in $\p3$. We denote by $H_{P,Q_d}$ the corresponding flag Hilbert scheme parametrizing pairs $(C,X)$ such that $C \in H_P, X \in H_{Q_d}$ and $C \subset X$.
A curve $C$ on a smooth surface in $\p3$ is said to be \emph{semi-regular} if $H^1(\mo_X(C))=0$. 
We prove that
\begin{thm}
 Let $X$ be a smooth surface in $\p3$ of degree $d$ and $C \subset X$, a semi-regular curve with Hilbert polynomial, say $P$. For any irreducible component, $L'$ of $\ov{\NL([C])}$
 (the closure is taken in $U_d$ under Zariski topology)
 there exists an irreducible component $H'$ of $H_{P,Q_d}$ containing the pair $(C,X)$ such that $\pr_2(H')_{\red}$ 
 coincides with $L'_{\red}$, where $\pr_2$ is 
 the second projection map from $H_{P,Q_d}$ to $H_{Q_d}$. In particular, if $C$ is reduced, connected and $d \ge \deg(C)+4$ then this holds true and the irreducible component $H'$ is uniquely determined by $L'$.
\end{thm}
See Theorem \ref{dim} and Lemma \ref{hf11} for the precise statements and its proof. Using this proof we further show that under the above 
bound on the codimension of $L$, the lattice will infact be generated by classes of lines (see Lemma \ref{gh4}). 
Finally, we do a computation in Proposition \ref{gh13}, to determine the ``arrangement'' of these lines which would 
help us determine the component with the correct codimension.

{\bf{Acknowledgement:}} I would like to thank Prof. R. Kloosterman for reading the preliminary version of this article and several helpful discussions. 

\section{Introduction to Noether-Lefschetz locus}

\begin{para}
 In this section we recall the basic definitions of Noether-Lefschetz locus. See \cite[\S $9, 10$]{v4} and \cite[\S $5, 6$]{v5} for a detailed presentation of the subject.
\end{para}

\begin{note}\label{n1}
By a \emph{component} of $\NL_d$, we mean an irreducible component. By a \emph{surface} we always mean a smooth surface in $\mathbb{P}^3$.
Denote by $Q_d$ the Hilbert polynomial of degree $d$ surfaces in $\p3$. Given, a Hilbert polynomial $P$, denote by $H_P$ the corresponding Hilbert scheme and by $H_{P,Q_d}$ the corresponding flag Hilbert scheme.
Also, for a point $u \in U_d$, denote by $X_u$ the fiber $\pi^{-1}(u)$.
\end{note}

\begin{note}\label{nl1}
 Let $X \in U_d$ and $\mo_X(1)$, the very ample line bundle on $X$ determined by the closed immersion $X \hookrightarrow \p3$ arising (as in \cite[II.Ex.$2.14$(b)]{R1}) from the graded homomorphism
of graded rings $S \to S/(F_X)$, where $S=\Gamma_*(\mo_{\p3})$ and $F_X$ is the defining equations of $X$.
Denote by $H_X$ the very ample line bundle $\mo_X(1)$. 
Note that a very ample line bundle on $X_u$ for any $u \in U$ remains very ample in the family $\mc{X}$, hence the corresponding cohomology class remains of type $(1,1)$ 
in $\mc{X}$.
\end{note}

     \begin{para}
     Let $X$ be a surface.
      The Lefschetz hyperplane section theorem implies, \[H^{2}(X,\mb{C})\cong H^{2}(X,\mb{C})_{\prim} \oplus \mb{C}H_X,\]
      where $H_X$ is the very ample line bundle on $X$ and $H^2(X,\mb{C})_{\prim}$ is the primitive cohomology. 
      This gives us a natural projection map from $H^{2}(X,\mb{C})$ to $H^{2}(X,\mb{C})_{\prim}$. For $\gamma \in H^{2}(X,\mb{C})$, denote by $\gamma_{\prim}$ the image of $\gamma$ under this morphism.
      Since the very ample line bundle $H_X$ remains of type $(1,1)$ in the family $\mc{X}$, we can therefore conclude that
        $\gamma \in H^{1,1}(X)$ remains of type $(1,1)$ if and only if $\gamma_{\prim}$ remains of type $(1,1)$. In particular,
      $\NL(\gamma)=\NL(\gamma_{\prim})$.
     \end{para}
     
\begin{defi} 
We now discuss the tangent space to the Hodge locus, $\NL(\gamma)$.
We know that the tangent space to $U$ at $X$, $T_XU$ is isomorphic to $H^0(\N_{X|\p3})$.
This is because $U$ is an open subscheme of the Hilbert scheme $H_{Q_d}$, the tangent space of which at the point $X$ is simply
$H^0(\N_{X|\p3})$. 
Given the variation of Hodge structure above, we have (by Griffith's transversality) the differential map:
\[\overline{\nabla}:H^{1,1}(X) \to \mathrm{Hom}(T_XU,H^2(X,\mathcal{O}_X))\]induced by the Gauss-Manin connection.
Given $\gamma \in H^{1,1}(X)$ this induces a morphism, denoted $\overline{\nabla}(\gamma)$ from $T_XU$ to $H^2(\mo_X)$.
The tangent space at $X$ to $\NL(\gamma)$ is then defined to be $\ker(\overline{\nabla}(\gamma))$.
\end{defi}

\begin{para}\label{gr0}
 The boundary map \[\rho:H^0(\N_{X|\p3}) \to H^1(\T_X)\] arising from the long exact sequence associated to the 
 short exact sequence:
 \[0 \to \T_X \to \T_{\p3}|_X \to \N_{X|\p3} \to 0\]
 is called the \emph{Kodaira-Spencer} map. The morphism $\overline{\nabla}(\gamma)$ is related to the Kodaira-Spencer map as we will see below.
\end{para}

\begin{para}\label{c9}
 Note that there exists a natural cup product morphism, 
 \[H^1(X,\T_X) \otimes H^{1}(X,\Omega^1_X) \xrightarrow{\bigcup} H^2(X,\mathcal{O}_X).\]
 For $\gamma \in H^1(\Omega_X^1)$ this induces a morphism, denoted $\bigcup \gamma$, from $H^1(\T_X)$ to $H^2(\mo_X)$.
 We then have the following result in Hodge theory (see \cite[Theorem $10.21$]{v4}):
\end{para}

\begin{lem}\label{a3}
The differential map $\overline{\nabla}(\gamma)$ conincides with the following:
\[T_XU\cong H^0(\N_{X|\p3}) \xrightarrow{\rho} H^1(\T_X) \xrightarrow{\bigcup \gamma} H^2(\mo_X).\]
\end{lem}

\section{Hodge locus and Hilbert flag schemes}

\begin{para}
 In this section we define what is a semi-regular map. We then briefly study Hodge locus for a family of smooth projective
 surfaces in $\p3$ and show how it is related to certain Hilbert flag schemes. More specifically, we shall study 
 the Hodge locus corresponding to certain effective algebraic cycles which will be semi-regular. For such classes we will see that
 the Hodge locus ``coincides'' with a component of a flag Hilbert scheme. We elaborate on the details in this section.
\end{para}

\subsection{Semi-regularity map and tangent space to Hodge locus}

\begin{para}
We start with the definition of a semi-regular curve.
 Let $X$ be a surface and $C \subset X$, a curve in $X$. Since $X$ is smooth, $C$ is local complete intersection in $X$.
 This gives rise to the short exact sequence:\[0 \to \mo_X(-C) \to \mo_X \to i_*\mo_C \to 0,\]where $i$ is the natural inclusion
 morphism from $C$ into $X$. Note that, $\mo_X(C)$ is locally free $\mo_X$-module, hence flat. Therefore, tensoring this short exact sequence by $\mo_X(C)$
 we get
 \begin{equation}\label{sh1}
 0 \to \mo_X \to \mo_X(C) \to \N_{C|X} \to 0 
                 \end{equation}                       
                 is exact,
 where $\N_{C|X}$ is the normal sheaf
 $\Hc_X(\mo_X(-C),i_*\mo_C)$ which is isomorphic to the sheaf $i_*\mo_C \otimes_{\mo_X} \mo_X(C)$ (see \cite[Ex. II.$5.1$(b)]{R1}). The \emph{semi-regularity map}
 is the morphism \[\pi:H^1(\N_{C|X}) \to H^2(\mo_X)\]which arises from the long exact sequence associated to the short
 exact sequence (\ref{sh1}). We say that $C$ is \emph{semi-regular} if $\pi$ is injective. 
\end{para}



\begin{para}
 The Lefschetz hyperplane section theorem implies that $H^1(\mo_X)=0$. Then, the long 
 exact sequence associated to (\ref{sh1}) contains the following segment:
 \[0 \to H^1(\mo_X(C)) \to H^1(\N_{C|X}) \xrightarrow{\pi} H^2(\mo_X).\]
 So, $H^1(\mo_X(C))=0$ is equivalent to $\pi$ being injective, hence $C$ being semi-regular. We now prove a result that would help
 us determine when a curve is semi-regular.
\end{para}

\begin{lem}\label{hf11}
Let $C$ be a connected reduced curve and $d \ge \deg(C)+4$ then $h^1(\mathcal{O}_X(C))=0$. In particular, $C$ is semi-regular.
\end{lem}

\begin{proof}
Since $X$ is a hypersurface in $\p3$ of degree $d$, $\I_X \cong \mo_{\p3}(-d)$.  Consider the short exact sequence:
\[0 \to \I_X \to \I_C \to \mo_X(-C) \to 0.\]
Tensoring this by $\mo_{\p3}(k)$, we get the following terms in the associated long exact sequence:
\[... \to H^1(\I_C(k)) \to H^1(\mo_X(-C)(k)) \to H^2(\I_X(k)) \to ...\]
Now, $H^2(\mo_{\p3}(k-d))=0$ (see \cite[Theorem $5.1$]{R1}) and $\I_{C}$ is $\deg(C)$-regular (see 
\cite[Main Theorem]{gi1}). 
So, $H^1(\I_C(k))=0$ for $k \ge \deg(C)$. This implies $H^1(\mo_X(-C)(k))=0$ for $k \ge \deg(C)$. 
By Serre duality, $0=H^1(\mo_X(-C)(d-4)) \cong H^1(\mo_X(C))$. So, $C$ is semi-regular.
\end{proof}

\begin{para}
Let $X$ be a surface and $C \subset X$ be a curve.
 We now do a computation to show that for $d \ge \deg(C)+4$, $\dim |C|=0$, where $|C|$ is the linear system of $C$ in $X$. 
\end{para}

\begin{lem}\label{a4e}
Let $d \ge 5$ and $C$ be an effective divisor on a smooth degree $d$ surface $X$ of the form $\sum_i a_iC_i$, where $C_i$ are integral curves with $\deg(C)+4 \le d$.
Then, $h^0(\N_{C|X})=0.$ In particular, $\dim |C|=0$, where $|C|$ is the linear system associated to $C$.
\end{lem}

\begin{proof}
Let $C=\sum_ia_iC_i$ with $C_i$ integral. Then, \[\deg((\mathcal{O}_X(C)|_C \otimes \mo_C)|_{C_i})=a_iC_i^2+\sum_{j \not= i} a_jC_i.C_j.\]
Denote by $e_i:=\deg(C_i)$. 
Using the adjunction formula and the fact that $K_X \cong \mathcal{O}_X(d-4)$, we have that 
\begin{eqnarray*}
\deg((\mathcal{O}_X(C)|_C \otimes \mo_C)|_{C_i})&=&2a_i\rho_a(C_i)-2a_i-(d-4)a_ie_i+\sum_{j \not= i}a_jC_i.C_j\\
&\le&a_i(e_i^2-(d-1)e_i)+\sum_{j \not=i}a_jC_iC_j\\
&\le&a_i(e_i^2-3e_i-e_i\sum_ja_je_j)+\sum_{j \not=i}a_je_ie_j.
\end{eqnarray*}
The first inequality follows from the bound on the genus of a curve in $\mathbb{P}^3$ in terms of its degree (see \cite[Example $6.4.2$]{R1}).
The second inequality follows from the facts that $d \ge \deg(C)+4$ and $C_i.C_j \le e_ie_j$. It then follows directly that $\deg((\mathcal{O}_X(C)|_C \otimes \mo_C)|_{C_i})<0$.
This implies that $h^0(C_i,(\mathcal{O}_X(C)|_C \otimes \mc{O}_C)|_{C_i})=0$ for all $i$. So, 
$h^0(\N_{C|X})=h^0(C,\mathcal{O}_X(C)|_C \otimes \mo_C)=0.$

Since $h^1(\mathcal{O}_X)=0$ (by Lefschetz hyperplane section Theorem) and $h^0(\mathcal{O}_X)=1$, using 
the long exact sequence associated to the short exact sequence 
\begin{equation}\label{eq2}
0 \to \mathcal{O}_X \to \mathcal{O}_X(C) \to \mathcal{O}_X(C)|_C \otimes \mo_C \to 0
\end{equation}
we get that $h^0(\mathcal{O}_X(C))=1$. Since $|C|=\mb{P}(H^0(\mo_X(C)))$, the lemma follows.
\end{proof}





\subsection{Flag Hilbert scheme and Hodge locus}

\begin{para}
 In this section we introduce the basic definitions of flag Hilbert schemes. See \cite[\S $4$]{S1} for further details. We then prove the main result of this section which relates Hodge locus to Hilbert schemes.
\end{para}

\begin{para}
Given an $m$-tuple of polynomials $\mathcal{P}(t)=(P_1(t),P_2(t),...,P_m(t))$, we define the  contravariant functor, called the \emph{Hilbert flag functor} 
relative to $\mathcal{P}(t)$,\[FH_{\mathcal{P}(t)}:(\mbox{schemes}) \to \mbox{sets}\]
\[S \mapsto \{(X_1,X_2,...,X_m)|X_1 \subset X_2 \subset ... \subset \mathbb{P}^3_S\}\]such that the Hilbert polynomial of $X_i$ is $P_i(t)$ and $X_i$ is an
$S$-closed subscheme of $X_{i+1}$. We call such an $m$-\emph{tuple a flag relative to} $\mathcal{P}(t)$.
\end{para}

\begin{para}
 The functor $FH_{\mathcal{P}(t)}$ is representable by a projective scheme, $H_{\mathcal{P}(t)}$ which parametrizes all such flags relative to $\mathcal{P}(t)$. We call this the \emph{Hilbert flag scheme}.
\end{para}

\begin{thm}\label{dim}
 Let $X$ be a surface, $C$ be a semi-regular curve in $X$ and $\gamma \in H^{1,1}(X,\mb{Z})$ be the class of the curve $C$. For any irreducible component $L'$ of $\overline{\NL(\gamma)}$
 (the closure is taken in the Zariski topology on $U_d$)
 there exists an irreducible component $H'$ of $H_{P,Q_d}$ containing the pair $(C,X)$ such that the associated reduced scheme $\pr_2(H')_{\red}$ 
 coincides with $L'_{\red}$, where $\pr_2$ is 
 the second projection map from $H_{P,Q_d}$ to $H_{Q_d}$. Furthermore, if $d \ge \deg(C)+4$ then such $H'$ is uniquely determined by $L'$.
\end{thm}

\begin{proof}
The first part of the theorem follows directly from \cite[Theorem $7.1$]{b1}. 

Furthermore, Lemma \ref{a4e} implies that $\dim |C|=0$. So, given an irreducible component, say $L'$ of $\overline{\NL(\gamma)}$ such that $X$ is a general
  element, there exists an unique irreducible component $H'$ of $H_{P,Q_d}$ containing the pair $(C,X)$ such that $\pr_2(H')_{\red}$ coincides with $L'_{\red}$. 
  This proves the rest of the theorem.

  
  \end{proof}


\section{Variation of lattices}
\begin{para}
 In this section we give a formula to compute the dimension of an irreducible component of a Hodge locus (see Proposition \ref{gh12}). 
 This result will be particularly useful to prove the asymptotic case of a Griffiths-Harris conjecture, which we see in the next section.
\end{para}

\begin{para}\label{pam6}
Let $X$ be a surface of degree $d$. 
An \emph{augmented lattice } $\Lambda_X$ \emph{ on } $X$ \emph{ of rank } $r$, is a rank $r$ $\mb{Z}$-submodule $\Lambda_X \subset H^2(X,\mb{Z})$ generated by the class of the very ample line bundle $H_X$ 
(as in \ref{nl1})
and cohomology classes of $r-1$ reduced curves, say $C_1,...,C_{r-1}$ such that
$\Lambda_X$ is saturated in the sense that for all $\lambda \in \Lambda_X, c \in \mb{Q}$ if $c\lambda \in H^2(X,\mb{Z})$ then $c\lambda \in \Lambda_X$. For such $\Lambda_X$, 
we say that $C_i$ for $i=1,...,r-1$ \emph{generate} $\Lambda_X$.
We say that $\Lambda_X$ is \emph{prime} if $C_1,...,C_{r-1}$ are integral.
\end{para}


\begin{para}\label{pam3}
Let $\Lambda_X$ be as in \ref{pam6}. 
We can define \[ \NL(\Lambda_X):=\{G \in U | \overline{\gamma}_G=0, \mbox{ for all } \gamma \in \Lambda_X\}.\]
For a surface $X$ and reduced curve $C \subset X$ such that the cohomology class $[C]$ of $C$ is not a $\mb{Q}$-multiple of $c_1(H_X)$, denote by $\Lambda_X^0$ the rank $2$ $\mb{Z}$-module generated by $[C]$ and 
$c_1(H_X)$, where $c_1$ is the first Chern class map.
Since a very ample line bundle remains of type $(1,1)$ is the family $\mc{X}$, $\NL(\Lambda_X^0)$ coincides with $\NL([C])$. More generally, for a rank $r$ augmented lattice $\Lambda_X$ generated by $C_1,...,C_{r-1}$ we have,  $\NL(\Lambda_X)$ is isomorphic to the fiber product $\NL([C_1]) \times_{H_{Q_d}} ... \times_{H_{Q_d}} \NL([C_{r-1}])$.
\end{para}

\begin{para}\label{pam2}
Let $\Lambda_X$ be as before of rank $2$, generated by a reduced curve, say $C$. 
Let $P$ be the Hilbert polynomial of $C$. Assume $d \ge \deg(C)+4$. Using Theorem \ref{dim} we can conclude that for general $X' \in \NL(\Lambda_X)$ there exists a curve $C' \subset X'$ such that 
$\overline{\NL([C'])}$ is an irreducible component of $\overline{\NL(\Lambda_X)}$ and $C'$ deforms to $C$, i.e., $C'$ has the same Hilbert polynomial $P$. Denote by $\Lambda_{X'}$ the augmented lattice on $X'$ of 
rank $2$ generated by $C'$.
Theorem \ref{dim} again implies that there exists an unique irreducible component, denoted $H_{\Lambda_{X'}}$, of $H_{P,Q_d}$ containing the pair $(C',X')$ such that $\pr_2(H_{\Lambda_{X'}})$ is isomorphic to $\overline{\NL(\Lambda_{X'})}$.
Denote by $L_{\Lambda_{X'}}:=\pr_1(H_{\Lambda_{X'}})$. From now on, we will always assume that $\overline{\NL(\Lambda_X)}$ is irreducible, which is equivalent to $X$ 
being general in $\NL(\Lambda_X)$, in particular,
away from the points of intersection of any two irreducible components of $\overline{\NL(\Lambda_X)}$. 
\end{para}

\begin{para}
Suppose now that $\Lambda_X$ is of rank $r$ generated by $C_1,...,C_{r-1}$. Let $P_1,...,P_{r-1}$ be the Hilbert polynomials of $C_1,...,C_{r-1}$, respectively. Consider the natural morphism  
 \[p:H_{P_1,Q_d} \times_{H_{Q_d}} ... \times_{H_{Q_d}} H_{P_{r-1},Q_d} \to H_{Q_d}.\] Assume $d \ge \sum_{i=1}^{r-1}\deg(C_i)+4$. Using Theorem \ref{dim}, we can conclude that for every irreducible component $L'$
 of $\overline{\NL(\Lambda_X)}$ there exists an unique irreducible component, say $H$ of $H_{P_1,Q_d} \times_{H_{Q_d}} ... \times_{H_{Q_d}} H_{P_{r-1},Q_d}$ containing $(C_1,X) \times ... \times (C_{r-1},X)$ 
 such that $p(H)$ coincides with $L'$. Similarly as in \ref{pam2}, by taking $X$ general in $\NL(\Lambda_X)$ we can ensure that $\overline{\NL(\Lambda_X)}$ is irreducible. 
 Denote by $H_{\Lambda_X}$ the irreducible component of $H_{P_1,Q_d} \times_{H_{Q_d}} ... \times_{H_{Q_d}} H_{P_r,Q_d}$ such that $p(H_{\Lambda_X})$ coincides with $\overline{\NL(\Lambda_X)}$. 
 Denote by $L_{\Lambda_X}:=\pr(H_{\Lambda_X})$, where $\pr$ is  the natural projection map from $H_{P_1,Q_d} \times_{H_{Q_d}} ... \times_{H_{Q_d}} H_{P_{r-1},Q_d}$ to $H_{P_1} \times ... \times H_{P_{r-1}}$.
\end{para}

\begin{prop}\label{gh12}
Let $r \ge 3$, $X$ be a surface of degree $d$ and $\Lambda_X$ be an augmented lattice of rank $r+1$ generated by $r$ reduced curves $C_1,...,C_r$. Assume that $\sum_{i=1}^r \deg(C_i)+4 \le d$.
Then, the dimension of $\ov{\NL(\Lambda_X)}$ is given by the following formula: \[\codim \ov{\NL(\Lambda_X)}=\codim I_d(C)-\dim L_{\Lambda_X}\] where $C=C_1' \bigcup ... \bigcup C_r'$ for a  general $r$-tuple $(C_1',...,C_r')$ in $L_{\Lambda_X}$
and $I_d(C)$ is the degree $d$ graded piece in the ideal $I(C)$ of $C$.
\end{prop}

\begin{proof}
Consider the diagram,
\[\begin{diagram}
H_{\Lambda_X}&\rTo^{\pr_1}&L_{\Lambda_X}\\
\dTo^{\pr_2}& \\
\ov{\NL(\Lambda_X)}\\
\end{diagram}\]
Denote by $P_i$ the Hilbert polynomial of curves $C_i$, respectively. Recall, $L_{\Lambda_X}$ is contained in $H_{P_1} \times ... \times H_{P_r}$. 
For an $r$-tuple $(C_1,...,C_r) \in L_{\Lambda_X}$, the fiber of $\pr_1$ parametrizes the space of smooth degree $d$ surfaces containing $C=C_1 \bigcup ... \bigcup C_r$, which is an open subscheme in $\mb{P}(I_d(C))$.
Since $I_d(C)$ is irreducible, the dimension of the generic fiber of $\pr_1$ is equal to $\dim I_d(C)-1$, where $(C_1,...,C_r) \in L_{\Lambda_X}$ is a general element.
The fiber of $\pr_2$ over $\pr_2((C_1,...,C_r,X))$ is isomorphic to $|C_1|\times ... \times |C_r|$. But, Lemma \ref{a4e} implies $\dim |C_i|=0$ for $i=1,...,r$.
So, the dimension of the generic fiber of $\pr_2$ is zero.
Then, \[\dim H_{\Lambda_X}=\dim L_{\Lambda_X}+\dim I_d(C)-1=\dim \NL(\Lambda_X).\]
\[\mbox{So, } \codim \overline{\NL(\Lambda_X)}=\dim \mb{P}(H^0(\mo_{\p3}(d)))-\dim \overline{\NL(\Lambda_X)}=h^0(\mo_{\p3}(d))-\dim I_d(C)-\dim L_{\Lambda_X}.\]
This finishes the proof of the proposition.
\end{proof}

\section{A Griffiths-Harris conjecture}

\begin{para}
We now come to the final section of the article, where we prove an asymptotic case of a Griffiths-Harris conjecture. Recall,
a \emph{Griffiths-Harris conjecture} states in \cite{GH} that: 
\begin{center}
For $3 \le r \le d$, the codimension of an irreducible component of $\NL_{r,d}$ is at least equal to \[(r-1)(d-3)-\binom{r-3}{2}.\]
Furthermore, there exists a component of $\NL_{r,d}$ of this codimension parametrizing the space of surfaces 
containing $r-1$ lines on the same plane. 
\end{center}
\end{para}

\begin{note}
We will denote by $\N_d(r)$ the number, \[(d-3)(r-1)-\binom{r-3}{2}.\]We now recall a result in Noether-Lefschetz locus due to Otwinowska which will help us characterize the irreducible components of $\NL_{r,d}$ with codimension
less than or equal to $\N_d(r)$.
\end{note}

\begin{thm}[{\cite[Theorem $1$]{ot}}]\label{gh22}
 Let $\gamma$ be an augmented lattice of rank $2$ on a degree $d$ surface. There exists $C \in \mathbb{R}_+^*$ depending only on $r$ such that for $d \ge C(r-1)^8$ if $\codim \NL(\gamma) \le (r-1)d$ 
 then $\gamma_{\prim}=\sum_{i=1}^ta_i[C_i]_{\prim}$, where $a_i \in \mathbb{Q}^*$, $C_i$ are reduced curves and $\deg(C_i) \le (r-1)$ for
 $i=1,...,t$ for some positive integer $t$.
\end{thm}

\begin{para}
Throughout this section we denote by $r$ an integer greater than or equal to $3$ and for a fixed $r$, denote by $d$, an integer as mentioned in Theorem \ref{gh22}. We will \emph{assume} 
that $d$ is at least $r^3$ which will be used only in a computation in Lemma \ref{gh4}. The other results do not have any restriction on $d$ in terms of $r$.
\end{para}

\begin{prop}\label{c3}
Let $L$ be an irreducible component of $\NL_{r,d}$. Then $L$ is locally homeomorphic to $\NL(\Lambda)$ for some prime augmented lattice $\Lambda$ of 
rank at least $r$ on a surface $X \in L$, general.
\end{prop}

\begin{proof}
Let $L \subset \NL_{r,d}$ be an irreducible component. Let $X$ be a general element in $L$. This implies that for the Picard lattice $\Lambda:=\NS(X)$,  $\NL(\Lambda)_{\red}$ is an open subscheme of $L_{\red}$.
We can assume that $\Lambda$ is a prime lattice. Since $X$ is an element in $\NL_{r,d}$, the rank of $\Lambda$ is greater than or equal to $r$.
\end{proof}

\begin{prop}\label{gh14}
If $\Lambda$ is an augmented prime lattice of rank $t$ on some degree $d$ surface and $\codim \overline{\NL(\Lambda)} \le (r-1)d$. 
Then there exists a prime lattice $\Lambda'$ of rank greater than or equal to $t$ generated by classes of curves of degree less than or equal to $r-1$,
such that $C_i$ deforms along $\NL(\Lambda')$ and
$\overline{\NL(\Lambda)}_{\red}=\overline{\NL(\Lambda')}_{\red}$.
\end{prop}

\begin{proof}
 
Let $X \in \NL(\Lambda)$.
There exists a maximal lattice $\Lambda' \subset H^2(X,\mb{Z})$ such that $\Lambda'$ remains of type $(1,1)$ in $\NL(\Lambda)$ i.e., $\NL(\Lambda)_{\red}=\NL(\Lambda')_{\red}$. 
Now, there exists 
a surface $X' \in \NL(\Lambda')$ such that the N\'{e}ron-Severi group $\mr{NS}(X')$ is the translate (under deformation from $X$ to $X'$) of 
$\Lambda'$ in $H^2(X',\mb{Z})$ which we again denote by $\Lambda'$ for convinience of notation. Then Theorem \ref{gh22} implies that  any $\gamma \in \Lambda'$ is of  the form $\sum_i a_i[C_i]+bH_X$ with $\deg(C_i) \le r-1$.
So, $\Lambda'$ can be generated by classes of curves of degree at most $r-1$ and the class of the very ample line bundle $H_X$. 
Now, the class of $[C_i]$ remains of type $(1,1)$ along $\NL(\Lambda')$. From Lemma \ref{hf11} it follows that $C_i$ is semi-regular. Then, \cite[Theorem $7.1$]{b1}
implies that the class of $[C_i]$ remains effective along $\NL(\Lambda')$.
This proves the proposition.
\end{proof}

\begin{para}
 We now recall a result due to Eisenbud and Harris which we use in the next lemma. Let $P$ be a Hilbert polynomial of a curve in $\p3$ of degree $e$ and $L$ be an irreducible component of $H_P$.
 The corollary after \cite[Theorem $1$]{eis} tells us that,
 \begin{thm}[\cite{eis}]\label{ei1}
  For $e>1$, the dimension of $L$ is less than or equal to $3+e(e+3)/2$.
   \end{thm}
\end{para}

\begin{lem}\label{gh4}
Let $\Lambda$ be a prime augmented lattice of rank $t+1$ on a degree $d$ surface, generated by irreducible curves $C_i$ for $i=1,...,t$ for some positive integer $t$
and $\deg(C_i) \le r-1$. Suppose $\codim \NL(\Lambda)\le (r-1)d$. Then, $\sum_{i=1}^t \deg(C_i) \le (r-1)$.
\end{lem}

\begin{proof}
We prove this by induction on $t$. This is trivially true for $t=1$.
Suppose this is true for all $t \le m$. 

Assume this is not true for $t=m+1$. In other words, there exists a prime lattice $\Lambda$ minimally generated by $m+1$ curves such that 
$\sum_i \deg(C_i)>r-1$. 
This implies (after rearranging the indices if necessary) there exists an integer $0 <t' \le m+1$ such that $C_1,...,C_{t'}$ satisfies $\sum_{i=1}^{t'} \deg(C_i)>(r-1)$ and $\sum_{i=1}^{t'-1} \deg(C_i) \le (r-1)$.
Then, $\sum_{i=1}^{t'} \deg(C_i) \le 2(r-1)$.
Denote by $P$ the Hilbert polynomial of the curve $C_1 \bigcup ... \bigcup C_{t'}$.
We replace $e$ by $2(r-1)$ in Theorem \ref{ei1} and conclude that the dimension of the Hilbert scheme $H_P$ 
is less than or equal to $3+(r-1)(2r+1)$. Using Proposition \ref{gh12}, the codimension 
of $\NL([C_1+...+C_{t'}])$ is greater than or equal to $\codim I_d(C_1 \bigcup ... \bigcup C_{t'})-\dim H_P$. Since $r-1<\deg(C_1+...+C_{t'})\le 2(r-1)$ and $d \ge r^3$, we get the following inequality using the upper bound on 
the arithmetic genus of a curve of degree less than or equal to $2(r-1)$:
\begin{eqnarray*}
 \codim \NL([C_1+...+C_{t'}]) & \ge &\codim I_d(C_1+...C_{t'})-\dim H_P\\
 &\ge& (rd-(2r-3)(2r-4)/2+1)-(3+(r-1)(2r+1))\\
 &=&(r-1)d+(d-(2r-3)(2r-4)/2-3-(r-1)(2r+1)+1)\\
 &>&(r-1)d
\end{eqnarray*}
contradicting the assumption.
\end{proof}

\begin{prop}\label{gh13}
Let $\Lambda$ be an augmented lattice of rank $r$ contained in a degree $d$ surface, generated by $l_i$ for $i=1,...,r-1$, where $l_i$ are lines for 
all $i$. Suppose $r \ge 3$. Then, $\codim \NL(\Lambda) \ge \N_d(r)$. Furthermore, if $l_i$ are on the same plane we have
an equality.
\end{prop}

\begin{proof}
We prove this by induction.
 Using Proposition \ref{gh12}, \[\codim \NL(\Lambda)=\codim I_d(\bigcup_{i=1}^{r-1} l_i)-\dim L_{\Lambda}.\]
 If $r=3$, $\codim \NL(\Lambda)=2d-6=\N_d(3)$.
 Assume the result holds true for all $r \le m$ for some integer $m \ge 4$. We now prove for $r=m+1$.
Denote by $\Lambda'$ the lattice generated by $l_i$ for $i=1,...,r-2$ and by $C$ the curve $\bigcup_{i=1}^{r-2} l_i$. 
Let $t:=\dim L_{\Lambda}-\dim L_{\Lambda'}$. Note that, $t \le 4$. 

Note that $l_{r-1}.C \le r-2$.
Denote by $\epsilon:=r-2-l_{r-1}.C$. Comparing with the list of values of $t$, we see 
\begin{enumerate}
 \item If $\epsilon=0$ then $l_{r-1}.C=r-2$ and $t \le 2$. In particular, for a fixed curve $C$ there is a $1$-$1$ correspondence between the 
 set of choices of $l_{r-1}$ intersecting $C$ in $r-2$ points and the set of planes $P$ interesting $C$ at $r-2$ distinct collinear points. If for a generic choice of $P$, $P \bigcap C$ are $r-2$ distinct collinear points
 then all lines in $C$ should lie on the same plane. In that case, $l_{r-1}$ intersect $C$ at $r-2$ points if and only if $l_{r-1}$ is on the same plane as $C$, hence $t=2$ (dimension of the space of lines in $\mb{P}^2$). 
 If this is not the case i.e., 
 a generic choice of $P$ does not intersect $C$ in $r-2$ distinct collinear points, then $t \le 2$.
\item $\epsilon=r-2$ and $t=4$ if $l_{r-1}$ does not intersect $C$.
\item $0 < \epsilon <r-2$ and $t \le 3$ otherwise.
\end{enumerate}

Now, 
\begin{eqnarray*}
 \codim \NL(\Lambda)&=&\codim I_d(C \bigcup l_{r-1})-\dim L_\Lambda\\
 &=&(\codim I_d(C)+\codim I_d(l_{r-1})-l_{r-1}.C)-\dim L_{\Lambda}\\
 &=&(\codim I_d(C)-\dim L_{\Lambda'})+(\codim I_d(l_{r-1})-t)-(r-2-\epsilon).
\end{eqnarray*}
Writing $t=4-(4-t)$ and using the induction step, we see that the right hand side of this equation is greater than
or equal to 
\begin{eqnarray*}
 ((r-2)(d-3)-\binom{r-4}{2})+(d-3)-(r-2-\epsilon)+(4-t)\\
 =((r-2)(d-3)-\binom{r-3}{2})+\epsilon-t+2
\end{eqnarray*}
Note that $\epsilon-t+2 \ge 0$ from the three cases considered above. Substituting this inequality gives us
$\codim \NL(\Lambda) \ge \N_d(r)$. 
 \end{proof}

\begin{thm}\label{gh6}
Let $L$ be an irreducible component of $\NL_{r,d}$. Then, $\codim L \ge (r-1)(d-3)-\binom{r-3}{2}$.
Furthermore, there exists a component $L$ of $\NL_{r,d}$ of this codimension parametrizing surfaces containing $r-1$ coplanar lines.
\end{thm}

\begin{proof}
It suffices to prove $L$ is locally of the form $\NL(\Lambda)$ for $\Lambda$ a rank $r-1$ prime lattice generated by $[l_i]$
for $i=1,...,r$, where $l_i$ are lines on the same plane. Using Proposition \ref{c3}, $L$ is of the form $\NL(\Lambda)$ for a rank $r-1$ prime
lattice $\Lambda$. Using Proposition \ref{gh14} we can assume that there exists a prime lattice $\Lambda'$ of rank $t$ 
greater than or equal to $r-1$ generated by classes of curves, say $C_i'$ for $i=1,...,t$ of degree less than or equal to $r-1$ such that 
$\NL(\Lambda)=\NL(\Lambda')$. Lemma \ref{gh4} implies that $t=r-1$ and $\deg(C_i')=1$ for all $i=1,...,r-1$.
Finally, the theorem follows from Proposition \ref{gh13}.
\end{proof}

\bibliographystyle{alpha}
 \bibliography{researchbib}
 
\end{document}